\documentclass[10pt,reqno]{amsart}
\usepackage[utf8]{inputenc}
\usepackage{amsfonts,amsthm}
\usepackage{amssymb,amsmath}
\usepackage{amsrefs}
\usepackage{xcolor}
\usepackage{appendix}
\usepackage{framed}
\usepackage{graphicx}
\usepackage{epstopdf}
\usepackage{geometry}
\usepackage{tikz}
\usepackage{txfonts}
\usepackage[normalem]{ulem}
\usepackage{hyperref}
\newtheorem{theorem}{Theorem}
\newtheorem{proposition}[theorem]{Proposition}
\newtheorem{definition}[theorem]{Definition}

\newtheorem{remark}[theorem]{Remark}
\newtheorem{lemma}[theorem]{Lemma}

\newtheorem{hypothesis}[theorem]{Hypothesis}

\newcommand{\R}{\mathbb{R}^d}

\newcommand{\rz}{\mathbb{R}}

\newcommand{\supp}{\hbox{supp}}
\newcommand{\E}{{\mathcal{E}_{\phi}}}

\newcommand{\M}{{\mathcal{L}}}

\newcommand{\HM}{\mathcal{H}}
\newcommand{\per}{{\rm Per}}
\newcommand{\rzz}{\mathbb{R}}
\newcommand{\uo}{u}
\newcommand{\X}{X_{\phi,\omega}}

\newcommand{\iphil}{I_{\ell}}
\newcommand{\tphil}{\mathbb{T}_{\ell}}
\definecolor{darkred}{rgb}{0.9,0.1,0.1}

\numberwithin{equation}{section}
\numberwithin{theorem}{section}
\begin{document}
\thispagestyle{empty}
\title[Symmetry of minimizers]{Symmetry of constrained minimizers of the Cahn-Hilliard energy on the torus}
\author{Michael Gelantalis}
\address{Michael Gelantalis, University of Tennessee}
\email{mgelanta@tennessee.edu}
\author{Alfred Wagner}
\address{Alfred Wagner, RWTH Aachen University}
\email{wagner@instmath.rwth-aachen.de}
\author{Maria G. Westdickenberg}
\address{Maria G. Westdickenberg, RWTH Aachen University}
\email{maria@math1.rwth-aachen.de}
\pagestyle{myheadings}
\markright{\sc }
\date{\today}
\begin{abstract}
We establish sufficient conditions for a function on the torus to be equal to its Steiner symmetrization and apply the result to volume-constrained minimizers of the Cahn-Hilliard energy. We also show how two-point rearrangements can be used to establish symmetry for the Cahn-Hilliard model.
In two dimensions, the Bonnesen inequality can then be applied to quantitatively estimate the sphericity
of superlevel sets. 
\end{abstract}
\maketitle
\large
\date{}
\section{Introduction}\label{intro}
We are interested in symmetry of constrained minimizers of a Cahn-Hilliard  energy on the
torus. Steiner symmetrization is a natural tool in such a setting, and it is easy to use Steiner
symmetrization to show that there exist minimizers with
the symmetries of the torus~\cite{GWW}. In this paper, we show that in fact any  constrained
minimizer is (up to a shift)
\emph{equal to its Steiner symmetrization}. To do so, we formulate general sufficient conditions
for a function on the torus to be equal to
its Steiner symmetrization. Applying the result to the Cahn-Hilliard model, we obtain in particular that
the superlevel sets of minimizers are simply connected.
In two dimensions, we use this together with the Bonnesen inequality to derive a new bound
on the sphericity of minimizers (cf.~Proposition~\ref{Asym.stronger}), which rules out
phenomena such as 'tentacles'.

An even simpler rearrangement is the two-point rearrangement or polarization of a function. In general two-
point rearrangements give weaker results than symmetrization. For the Cahn-Hilliard problem, however, we
will obtain from two-point rearrangements that a minimizer is equal to its reflection with respect to some
hyperplane and from here deduce strict monotonicity properties.

A fine analysis by Cianchi and Fusco \cite{CF} gives sufficient conditions under which equality of the
Dirichlet integrals (cf.~\eqref{inteq}, below) for nonnegative functions in $W_0^{1,1}$ on suitable
domains $\Omega\subset\R$ implies that the function equals its Steiner symmetrization; see
\cite[Theorem 2.2 and Section 1]{CF}. Their main assumption, which we will also require, is given by
\eqref{keyass1} below. When one replaces the condition of nonnegativity and zero boundary condition by a
periodic boundary condition, however, one encounters an additional degeneracy; for instance the function
whose graph is depicted
in Figure~\ref{bumps} is not equal to its Steiner symmetrization even though it satisfies
\eqref{keyass1} and \eqref{inteq}. We will show that \eqref{mM} below suffices to rule out such
counterexamples.

Throughout the paper we will use the notation $\iphil:=\left[-\ell, \ell\right]$
where the endpoints are identified and use $\tphil$ and $\tphil'$ to represent the $d$- and
$d-1$-dimensional tori
\begin{eqnarray*}
\tphil:=\underbrace{\iphil\times\hdots\times\iphil}_{d\: -\: times}\qquad\hbox{and}\qquad
\tphil':=\underbrace{\iphil\times\hdots\times\iphil}_{d-1\:-\: times}.
\end{eqnarray*}
We will often represent a point $x\in\tphil$ by $x=(x',y)$ where
$x' =(x_1,\hdots x_{d-1})\in\tphil'$ and  $y\in\iphil$.

The space $C^1(\tphil)$ will denote the space of continuous functions that are continuously
differentiable and $2\ell$ - periodic in each variable. For $x'\in\tphil'$ we define
\begin{eqnarray*}
m(x'):=\min\{u(x',y)\::\:y\in \iphil\}\qquad\hbox{and}\qquad M(x'):=\max\{u(x',y)\::\:y\in \iphil\}.
\end{eqnarray*}
We will establish our main  result for Steiner symmetrization under the following hypothesis
(which we will later show to hold true in the Cahn-Hilliard model).
\begin{hypothesis}\label{hyp:1}
For all $x'\in\tphil'$ there holds
\begin{align}
m(x')<M(x')\label{mM}
\end{align}
and
\begin{align}
\M^d\left(\left\{(x',y)\in\tphil\colon \partial_y u(x',y)=0\;\text{and}\; m(x')<u(x',y)
<M(x')\right\}\right)=0.\label{keyass1}
\end{align}
\end{hypothesis}
As a consequence of \eqref{keyass1}, we observe that
\begin{align}
\M^{1}\left(\left\{y\in\iphil:\partial_{y}u(x',y)=0,\: m(x')<u(x',y)<M(x')\right\}\right)=0\label{keyass3}
\end{align}
for a.e. $x'\in\tphil'$.
According to Lemma \ref{equimeas} below, the same holds true for the Steiner symmetrization.
We will use these facts later in making use of the Coarea Formula.
\begin{figure}
\begin{center}
\scalebox{.8}{\includegraphics{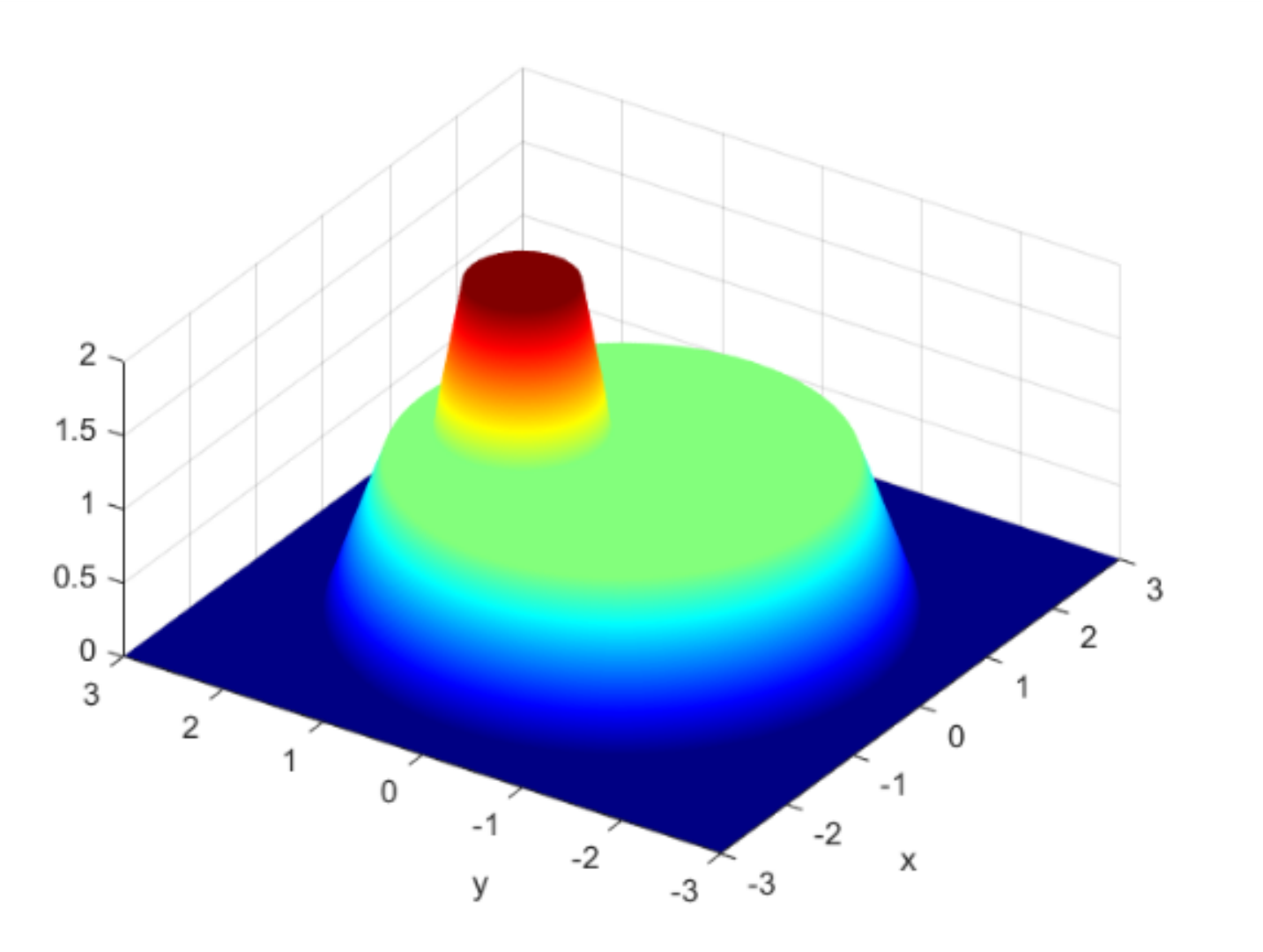}}
\end{center}
\caption{This function fulfills \eqref{inteq} but violates \eqref{keyass1}. Sliding the ```top layer'' of the layer-
cake around on the upper plateau does not change its Dirichlet energy and the function is not equal to its
Steiner symmetrization.}\label{cake}
\end{figure}
\begin{figure}
\begin{center}
\scalebox{.8}{\includegraphics{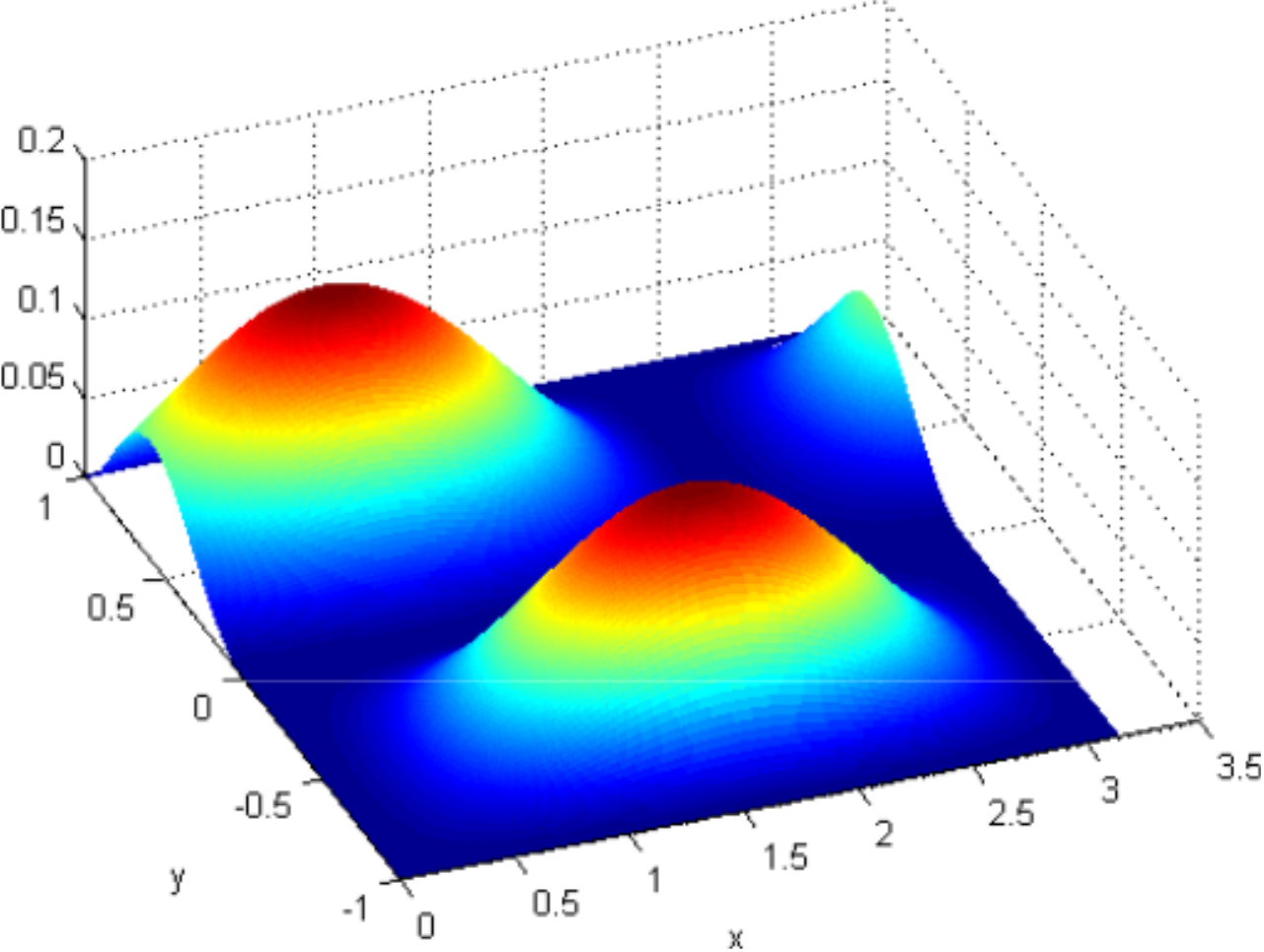}}
\end{center}
\caption{Steiner symmetrizing with respect to $x=\pi/2$ aligns the bumps without changing the Dirichlet
energy. This is possible because the function is constant on (merely) the line $y=0$. Such counterexamples
are ruled out by \eqref{mM}.}\label{bumps}
\end{figure}

Our main result for the Steiner symmetrization is the following.
\begin{theorem}\label{steinerthm}
Let $u\in C^{1}(\tphil)$ and assume Hypothesis~\ref{hyp:1} holds. If
\begin{align}\label{inteq}
\int\limits_{\tphil}\vert\nabla u\vert^2\:dx=\int\limits_{\tphil}\vert\nabla u^{s}\vert^2\:dx,
\end{align}
where $u^s$ represents the Steiner symmetrization of $u$ with respect to $\{y=0\}$, then there
exists $\beta\in \iphil$ such that $u=u^s(\cdot,y-\beta)$.
\end{theorem}
We now explain the Cahn-Hilliard model of interest. We consider the energy
\begin{eqnarray*}
\E=\int\limits_{\tphil}\frac{\phi}{2}\vert Du\vert^2+\frac{1}{\phi}G(u)\:dx
\end{eqnarray*}
where $G\in C^2(\mathbb{R})$ is a nonnegative double-well potential with zeros at $\pm 1$.
The canonical potential is
\begin{align*}
  G(s)=\frac{1}{4}(1-s^2)^2.
\end{align*}
We will denote the mean of a function by
\begin{eqnarray*}
\overline{u}:=\frac{1}{\vert \tphil\vert}\int\limits_{\tphil}u(x)\:dx
\end{eqnarray*}
and, for a smooth, monotone function $\zeta:\mathbb{R}\to [0,1]$ such that
\begin{align*}
\zeta(s)&=
\begin{cases}
1\quad&\text{for}\quad s\geq 1-\phi^{1/3}\\
0\quad&\text{for}\quad s\leq 1-2\phi^{1/3},
\end{cases}
\end{align*}
we will refer to
\begin{align*}
  \int\limits_{\tphil}\zeta(u(x))\:dx
\end{align*}
as the ``volume'' of the function $u$. We remark for future reference that the minimizers studied
in \cite{GWW} satisfy
\begin{eqnarray}\label{supp}
\supp(\zeta'(u))\ne\tphil.
\end{eqnarray}
We will always assume that \eqref{supp} holds.
The energy $\E$ is studied
on the set of $2\ell$-periodic functions with fixed mean and volume:
\begin{eqnarray*}
\X:=\left\{u\in W^{1,2}(\tphil)\::\:\overline{u}=-1+\phi,\:
\int\limits_{\tphil}\zeta(u)\:dx=\omega\right\}
\end{eqnarray*}
in the regime
 \begin{align}
\ell:=\frac{\phi L}{2},\qquad \phi=\xi\,L^{-d/(d+1)}.\label{regime}
 \end{align}
Minimizers of the energy $\E$ over $\X$ are known to exist and to satisfy quantitative
estimates that for $\phi$ small measure their closeness to certain sharp-interface ``droplet''
functions that are equal to $+1$ in a sphere and $-1$ on the complement.

Our main result for constrained minimizers of the Cahn-Hilliard energy is the following.
\begin{theorem}\label{t:ssym}
Let $u$ minimize $\E$ over $\X$ and assume that \eqref{supp} holds. Then there exists $
\beta^\star=(\beta_1,\ldots,\beta_d)\in\tphil$ such that $u$ is equal to its iterated Steiner
symmetrization with respect to $x_1=\beta_1$, $x_2=\beta_2,\ldots,\, x_d=\beta_d$.
\end{theorem}
\begin{remark}\label{rem:nonuq}
The theorem does not establish uniqueness; it does not rule out existence of
more than one Steiner symmetric constrained minimizer with prescribed volume $\omega$.
\end{remark}
\begin{remark}
Since the superlevel sets of Steiner symmetrizations are simply connected
(cf.~Remark~\ref{rem:prop} below), Theorem \ref{t:ssym} allows us to control the sphericity of
constrained minimizers in $d=2$ via the Bonnesen inequality; we will explain this result in
Subsection~\ref{ss:bonn}.
\end{remark}
Alternatively to Steiner symmetrization, one can use two-point rearrangements and apply a
Gidas-Ni-Nirenberg argument to the Cahn-Hilliard problem; in Section~\ref{S:GNN} we apply
this method to derive an alternative proof of Theorem \ref{t:ssym}.
\medskip

{\bf Organization}
In Section \ref{S:steiner} we prove Theorem \ref{steinerthm}. In Section \ref{S:steinerCH},
we apply this result to deduce a first proof of Theorem~\ref{t:ssym} and in
Subsection~\ref{ss:bonn}, we explain how this leads to a new bound on the sphericity of
minimizers in $d=2$. Then in Section \ref{S:GNN} we derive an alternate proof of
Theorem~\ref{t:ssym} using two-point rearrangements.
\section{Steiner Symmetrization on the Torus}\label{S:steiner}
Symmetrization techniques have been widely used to establish symmetry of global minimizers
of various energies (see for instance \cites{BL,K,LN}).
We mention in addition the continuous symmetrization of Brock (cf. \cite{B} and the references
therein), which he has used in some settings to establish symmetry of local minimizers.

When uniqueness of a minimizer is known a priori, this fact can often be used to deduce its
symmetry. When uniqueness is not assured, it becomes important to discuss the case that the energy of a given function equals that of its symmetrization. For Dirichlet-type functionals and
Schwarz symmetrization, this has been done in \cite{BZ}; for Steiner symmetrization, the first
sufficient conditions for equality go back to \cite{K}, and sharp conditions for nonnegative Sobolev
functions satisfying zero Dirichlet boundary conditions were presented recently in \cite{CF}.
Here we consider smooth functions on the torus.

We begin by recalling the definition and properties of Steiner symmetrization. In Subsection
\ref{steq2} we collect facts about the regularity of the distribution function. Finally in Subsection
\ref{ss:steinerresult} we prove Theorem~\ref{steinerthm}.
\subsection{Definitions}
We will occasionally use the notation
\begin{eqnarray*}
\iphil^{-}:=\left[-\ell, 0\right]
\quad\hbox{and}\quad
\iphil^{+}:=\left[0,\ell\right].
\end{eqnarray*}
For a compact set $C\subseteq\tphil$ and $x'\in\tphil'$ let
$
C(x') := \{y \in\iphil\/ : \:(x',y) \in C\}
$
and let $C'$ be the set of all $x'\in\tphil'$, such that $C(x')\neq\emptyset$.
\begin{definition}[Steiner symmetrization of a set]
We denote the \uline{Steiner symmetrization of $C$ with respect to the hyperplane
$\{y = 0\}$} by $C^s$, defined as
\begin{align*}
C^{s}:=\bigcup_{x'\in C'}C^s(x')
\quad\text{where}\quad
 C^s(x')=
\left\{(x',y)\in\tphil\::\:0\leq\vert y\vert\leq\frac{1}{2}\M^1(C(x'))\right\}
\end{align*}
and analogously for $C^s_i$, the symmetrization with respect to $\{x_i=0\}$.
\end{definition}
By construction $\vert C\vert=\vert C^s\vert$; we will refer to this property as the
\emph{equimeasurability of Steiner symmetrization}.
For a function $u:\tphil\to\mathbb{R}$ and $t\in\mathbb{R}$ we denote the superlevel
set of $u$ by
\begin{align*}
\Omega_t:=\left\{x\in\tphil\colon u(x)> t \right\}.
\end{align*}
\begin{definition}[Steiner symmetrization of a function]
We define the \uline{Steiner symmetrization of $u$ with respect to the hyperplane $\{y = 0\}$}
by $u^s$, defined as
\begin{eqnarray*}
u^{s}(x):=\sup\{t\in\rzz\::\:x\in (\Omega_t)^s\}
\end{eqnarray*}
and analogously
for $u^s_i$, the symmetrization with respect to  $\{x_i=0\}$.
Moreover we define the one-dimensional \uline{distribution function} of $u$
for $x'\in\tphil'$ and level $t\in\rzz$ as
\begin{eqnarray}
\mu_{u}(x',t):=\M^1\left(\{y\in \iphil\::\:u(x',y)>t\}\right).\label{distrib}
\end{eqnarray}
\end{definition}
The equimeasurability implies in particular that $\mu_{u}(x',t)=\mu_{u^s}(x',t)$.
\begin{definition}[Iterated Steiner symmetrization]
We denote the \uline{iterated Steiner symmetrization of $u$} by $u^\star$, defined via
symmetrizing first with respect to $\{y=x_d=0\}$, then $\{x_{d-1}=0\}$ through $\{x_1=0\}$.
\end{definition}
\begin{remark}\label{comm}
Iterating the Steiner symmetrization in a different order can give different results; see Figure \ref{figure2} for an example.
\begin{figure}
\begin{tikzpicture}[scale=1.0]
\begin{scope}[shift={(-4.25,0)}]
\draw [black, thick] (-1,0) -- (3,0);
\draw [black, thick] (1,0) -- (1,4);
\draw [black, very thick] (-1,0) -- (3,4) -- (1,0) -- (-1,0);
\draw (-1,0) -- (-1,-0.1);
\draw (1,0) -- (1,-0.1);
\draw (3,0) -- (3,-0.1);
\draw (1,2) -- (0.9,2);
\draw (1,4) -- (0.9,4);
\node [left] at (1,4) {\scriptsize{$2$}};
\node [left] at (1,2) {\scriptsize{$1$}};
\node [below] at (-1,0) {\scriptsize{$-1$}};
\node [below] at (1,0) {\scriptsize{$0$}};
\node [below] at (3,0) {\scriptsize{$1$}};
\node [below] at (1,-0.8) {\small{(a) The triangle $\Omega$}};
\end{scope}

\draw [black, thick] (-0.5,2) -- (2.5,2);
\draw [black, thick] (1,0) -- (1,4);
\draw [black, very thick] (0,2) -- (1,4) -- (2,2) -- (1,0) -- (0,2);
\node [below left] at (0.2,2) {\scriptsize{$-0.5$}};
\node [below right] at (1.8,2) {\scriptsize{$0.5$}};
\node [left] at (1,4) {\scriptsize{$1$}};
\node [left] at (1,0) {\scriptsize{$-1$}};
\node [below] at (1,-0.8) {\small{(b) $S_{2}\circ S_{1}(\Omega)$}};
\draw (0,2) -- (0,1.9);
\draw (2,2) -- (2,1.9);
\draw (1,4) -- (0.9,4);
\draw (1,0) -- (0.9,0);

\begin{scope}[shift={(4.25,0)}]
\draw [black, thick] (-1,2) -- (3,2);
\draw [black, thick] (1,0.5) -- (1,3.5);
\draw [black, very thick] (-1,2) -- (1,3) -- (3,2) -- (1,1) -- (-1,2);
\node [below] at (-1,2) {\scriptsize{$-1$}};
\node [below] at (3,2) {\scriptsize{$1$}};
\node [above left] at (1,2.8) {\scriptsize{$0.5$}};
\node [below left] at (1,1.17) {\scriptsize{$-0.5$}};
\node [below] at (1,-0.8) {\small{(c) $S_{1}\circ S_{2}(\Omega)$}};
\draw (-1,2) -- (-1,1.9);
\draw (3,2) -- (3,1.9);
\draw (1,1) -- (0.9,1);
\draw (1,3) -- (0.9,3);
\end{scope}
\end{tikzpicture}
\caption{The effect of repeated Steiner symmetrization depends on the order.}\label{figure2}
\end{figure}
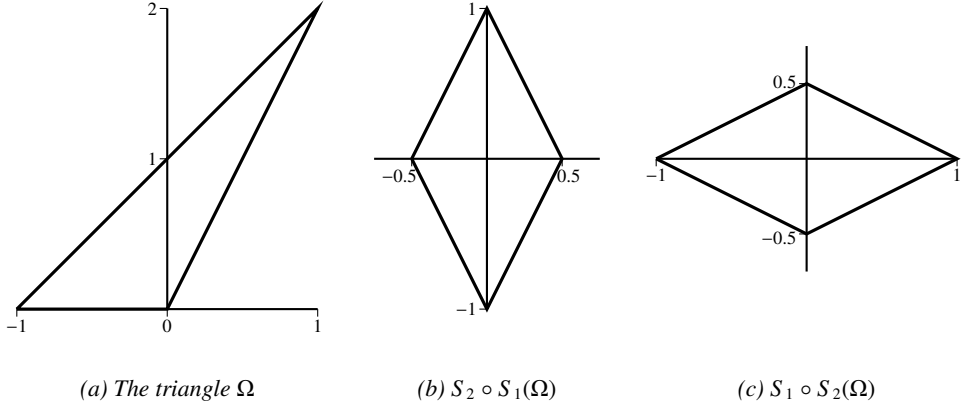
\end{remark}
\begin{remark}\label{rem:prop}
By construction, $u^s$ and $u^\star$ have the following properties:
\begin{itemize}
\item[(i)] $u^s(x',y)=u^s(x',-y)$ and similarly for $u^s_i$;
\item[(ii)] $\partial_{y}u^s(x',y)\leq 0$ on $\tphil'\times I_\ell^+$ and similarly for $u^s_i$;
\item[(iii)] the superlevel sets of $u^\star$ are simply connected and starshaped with respect to
the origin.
\end{itemize}
\end{remark}
In  \cite[Theorem 2.31]{K} it was proved that Steiner symmetrization on the torus does not
increase energy in the sense that
\begin{eqnarray}
\label{rearr1}&&\int\limits_{\tphil}\vert\nabla u^s\vert^2\:dx
\leq\int\limits_{\tphil}\vert\nabla u\vert^2\:dx\\
\label{rearr2}&&\int\limits_{\tphil}G(u^s)\:dx
=\int\limits_{\tphil}G(u)\:dx\qquad \text{for measurable functions $G$}.
\end{eqnarray}
We are interested in the question of when equality in \eqref{rearr1} implies $u=u^s$ (up to a
shift).
\subsection{Regularity of the distribution function}\label{steq2}
In this section we consider the one
dimensional distribution $\mu_{u}(x',t)$
for $x'\in\tphil'$ and $t\in\rzz$. We will use regularity of the distribution function in the next subsection for the proof
of Theorem \ref{steinerthm}.

Clearly $\mu_u$ is measurable both in $x'$ and $t$. Even if $u$ is smooth, however, the
function $\mu_{u}(\cdot,t)$ need not be continuous. For $x'\in\tphil'$ fixed and without assuming smoothness of $u$, our first lemma considers right- and left-continuity of the
distribution function in $t$. In particular, one observes that $\mu_{u}(x',\cdot)$ is continuous if
and only if $\M^1\left(\left\{y\in\iphil: u(x',y)=t\right\}\right)=0$ for all $t$.
\begin{lemma}\label{distri}
Let $u(x',\cdot)$ be measurable for each $x'\in\tphil'$. Then for all $x'\in\tphil'$ the distribution
function $\mu_{u}(x',\cdot)$ is right-continuous in the sense that
\begin{eqnarray}\label{distri1}
\lim_{\delta\downarrow 0}\left(\mu_{u}(x',t+\delta)-\mu_{u}(x',t)\right)=0
\end{eqnarray}
for all $x'\in\tphil'$. Moreover, we have
\begin{eqnarray}\label{distri2}
\lim_{\delta\downarrow 0}\left(\mu_{u}(x',t-\delta)-\mu_{u}(x',t)\right)=\M^1\left(\left\{y\in\iphil: u(x',y)=t\right\}\right)
\end{eqnarray}
for all $x'\in\tphil'$.
\end{lemma}
Formulas \eqref{distri1} - \eqref{distri2} can be derived from Section 2 in \cite{Ta}. We now seek
additional information about the regularity of $\mu_u$.
The proof of the next lemma follows via a mild adaptation of the proof of the BV regularity from \cite[Lemma 4.1]{CF}.
\begin{lemma}\label{mubv}
Let $u\in W^{1,1}(\tphil)$ and consider $\mu_u$ given by \eqref{distrib}
for $x'\in\tphil'$ and $t\in(m(x'),M(x'))$. There holds
\begin{eqnarray*}
\mu_{u}(\cdot,\cdot)\in BV(\tphil'\times\rzz).
\end{eqnarray*}
\end{lemma}
Lemma \ref{mubv} implies the existence of weak partial
derivatives of $\mu_{u}$; we refer to \cite[Section 1.7.2 and Theorem 4 of Section 6.1.3]{EG}
and \cite[Lemma 4.1]{CF}, where the explicit form of the partial derivatives was computed. We
summarize the result in the following proposition.
For the rest of the section we will assume that $u\in C^1(\tphil)$.
\begin{proposition}\label{mureg}
Let $u\in C^{1}(\tphil)$ and assume that \eqref{keyass1} holds.
Then the following formulas hold for $\mu_{u}$.
\begin{itemize}
\item[(i)] For $\M^{d-1}$ a.e. $x'\in\tphil'$, $\mu_{u}(x',\cdot)$ is differentiable for $\M^1$ a.e.
$t\in(m(x'),M(x'))$ and
\begin{eqnarray}
\label{mureg1}\qquad\partial_{t}\mu_{u}(x',t)=-\int\limits_{\{y\in\iphil:u(x',y)=t\}}
\frac{1}{\vert\partial_{y}u(x',y)\vert}\:d\HM^{0}(y).
\end{eqnarray}
\item[(ii)] For $\M^{d-1}$ a.e. $x'\in\tphil'$ and $\M^{1}$ a.e. $t\in(m(x'),M(x'))$, $\mu_{u}(\cdot,t)$ is
differentiable w.r.t. $x'$ and
\begin{eqnarray}
\label{mureg2}\qquad\partial_i\mu_{u}(x',t)=\int\limits_{\{y\in\iphil:u(x',y)=t\}}
\frac{\partial_iu(x',y)}{\vert\partial_{y}u(x',y)\vert}\:d\HM^{0}(y)
\end{eqnarray}
for all $i\in\{1,\ldots,d-1\}$.
\end{itemize}
\end{proposition}
For arbitrary  $x'\in\tphil'$ we now decompose the set
\begin{eqnarray*}
\lefteqn{\left\{y\in\iphil : t<u(x',y)<M(x')\right\}}\\
&=&\left\{y\in\iphil : \partial_{y}u(x',y)\ne 0, t<u(x',y)<M(x')\right\}\\
&&\cup
\left\{y\in\iphil : \partial_{y}u(x',y)=0, t<u(x',y)<M(x')\right\}.
\end{eqnarray*}
We set
\begin{eqnarray}
\mu_{u}^{reg}(x',t):=\M^1\left(\left\{y\in\iphil : \partial_{y}u(x',y)\ne 0, t<u(x',y)<M(x')\right\}\right)
\label{g}
\end{eqnarray}
and
\begin{eqnarray}
\mu_{u}^{sing}(x',t):=\M^1\left(\left\{y\in\iphil : \partial_{y}u(x',y)=0, t<u(x',y)<M(x')\right\}\right).
\label{h}
\end{eqnarray}
In the following remark we observe that $\mu_{u}^{reg}$ is more regular than $\mu_{u}$.
In particular we
get a pointwise  $t$ - derivative of $\mu_{u}^{reg}$.
\begin{remark}\label{regdist}
For any $x'\in\tphil'$ and $t\in (m(x'),M(x'))$,
we obtain from the Coarea Formula (see e.g.\ \cite[chapter 3.4]{EG} or \cite[chapter 2.12]
{AFP}) that the function $\mu_{u}^{reg}$ from \eqref{g} satisfies
\begin{eqnarray}\label{regdist1}
\mu_{u}^{reg}(x',t)
=
\int\limits_{t}^{M(x')}\int\limits_{\{y\in \iphil : u(x',y)=s\}}\frac{1}{\vert \partial_{y}
u(x',y)\vert}\:d\HM^{0}(y)\:ds,
\end{eqnarray}
and hence that
\begin{eqnarray*}
\partial_{t}\mu_{u}^{reg}(x',t)&=&-\int\limits_{\{y\in \iphil : u(x',y)=t\}} \frac{1}
{ \vert\partial_{y}u(x',y)\vert } \:d\HM^{0}(y).
\end{eqnarray*}
\end{remark}
For $\mu_{u}^{sing}$ we obtain the following result (see Lemma 2.4 in \cite{CF02}).
\begin{lemma}\label{singu}
Let $u\in C^1(\tphil)$. For any $x'\in\tphil'$ and any $t\in (m(x'),M(x'))$ the function
$\mu_{u}^{sing}$ defined in \eqref{h}
is nonincreasing and right-continuous in $t$ and satisfies $\partial_{t}\mu_{u}^{sing}(x',t)=0$ for
$\M^1$ - almost all $t\in (m(x'),M(x'))$.
\end{lemma}
A main point for us is that \eqref{keyass3} implies that
$\mu=\mu_{reg}$ for almost all $x'\in\tphil'$.
\subsection{Sufficient condition for equality of Dirichlet-energy on the torus}
\label{ss:steinerresult}
We will now show that the proof from \cite{CF} can be adapted under Hypothesis \ref{hyp:1} for
functions on the torus. We assume $u\in C^1(\tphil)$ since this is the case in our application
and elements of the proof simplify.

Before turning to the proof of Theorem~\ref{steinerthm}, we need to link ``plateaus'' of $u$ with
those of the symmetrization. The next lemma is the analogue of \cite[Proposition 2.3]{CF}; it simplifies in the $C^1$ setting but we omit the proof since the difference is not significant.
\begin{lemma}\label{equimeas}
Let $u\in C^1(\tphil)$. Then for all $x'\in\tphil'$ and all $t\in (m(x'),M(x'))$ we have
\begin{eqnarray}
\lefteqn{\M^1\left(\left\{y\in\iphil:\partial_{y}u(x',y)=0,\:t<u(x',y)<M(x')\right\}\right)}\notag\\
&& =
\M^1\left(\left\{y\in\iphil:\partial_{y}u^{s}(x',y)=0,\:t<u^{s}(x',y)<M(x')\right\}\right).
\label{eqeqmeas}
\end{eqnarray}
\end{lemma}
With this lemma in hand, we turn to the proof of the main theorem.
\begin{proof}[Proof of Theorem~\ref{steinerthm}]
We denote by $N$ a set with $\M^{d-1}(N)=0$
such that \eqref{keyass3} and \eqref{mureg1} - \eqref{mureg2} hold for all $x'\in\tphil'\setminus
N$.

\underline{Step 1.}[Derivatives of the distribution function in terms of the Steiner
symmetrization.] In light of \eqref{keyass3} and \eqref{eqeqmeas}, the one dimensional Coarea
Formula applied to the Dirichlet integral of $u^s$ gives
\begin{eqnarray}\label{coarea}
\int\limits_{\tphil}\vert\nabla u^s\vert^2\:dx=\int\limits_{\tphil'}\int\limits_{m(x')}^{M(x')}
\int\limits_{\{y\in\iphil:u^s(x',y)=t\}}
\frac{\vert\nabla u^s\vert^2}{\vert\partial_{y}u^s\vert} \:d\HM^{0}(y)\:dt\:dx'.
\end{eqnarray}
The equimeasurability of Steiner symmetrization and Proposition \ref{mureg} imply
\begin{eqnarray}
\lefteqn{-\int\limits_{\{y\in\iphil:u^s(x',y)=t\}}\frac{1}{\vert\partial_{y}u^s(x',y)\vert}
\:d\HM^{0}(y)}\notag\\
&\overset{\eqref{mureg1}}=&\partial_{t}\mu_{u^s}(x',t)\notag\\
&=&\partial_{t}\mu_{u}(x',t)=-\int\limits_{\{y\in\iphil:u(x',y)=t\}}\frac{1}{\vert\partial_{y}
u(x',y)\vert}\:d\HM^{0}(y)\label{distr1}
\end{eqnarray}
for all $x'\in\tphil'\setminus N$ and $\M^1$ a.e. $t\in(m(x'),M(x'))$. Similarly, there holds
\begin{eqnarray}
\lefteqn{\int\limits_{\{y\in\iphil:u^s(x',y)=t\}}\frac{\partial_{i}u^s(x',y)}{\vert\partial_{y}
u^s(x',y)\vert}\:d\HM^{0}(y)}\notag\\
&\overset{\eqref{mureg2}}=&\partial_{i}\mu_{u^s}(x',t)\notag\\
&=&\partial_{i}\mu_{u}(x',t)=
\int\limits_{\{y\in\iphil:u(x',y)=t\}}\frac{\partial_{i}u(x',y)}{\vert\partial_{y}u(x',y)\vert}
\:d\HM^{0}(y)\label{distr2}
\end{eqnarray}
for all $i\in\{1,\ldots,d-1\}$, $x'\in\tphil'\setminus N$, and $t\in(m(x'),M(x'))$.
Using that $u^{s}$ is symmetric and---because of Lemma \ref{equimeas}---satisfies
\eqref{keyass1}, we simplify the  left-hand sides of \eqref{distr1} and \eqref{distr2}
to deduce the formulas
\begin{eqnarray}
\label{distr3}\partial_{t}\mu_{u}(x',t)&=&-\frac{2}{\vert\partial_{y}u^{s}\vert}\Big
\vert_{\{y\in\iphil:u^{s}(x',y)=t\}}
\end{eqnarray}
and
\begin{eqnarray}
\label{distr4}\partial_{i}\mu_{u}(x',t)&=&\frac{2\partial_{i}u^{s}}{\vert\partial_{y}u^{s}\vert}
\Big\vert_{\{y\in\iphil:u^{s}(x',y)=t\}}
\end{eqnarray}
for all $i\in\{1,\ldots,d-1\}$, and for all $x'\in\tphil'\setminus N$ and $\M^1$-a.e. $t\in(m(x'),M(x'))
$.

\underline{Step 2.}[Cauchy-Schwarz and isoperimetric arguments.] For  $x'\in\tphil'\setminus N$ and the $\M^1$-a.e. $t\in(m(x'),M(x'))$ identified in Step 1, we use formulas
\eqref{distr3} - \eqref{distr4} to express
\begin{eqnarray*}
&&\int\limits_{\{y\in\iphil:u^{s}(x',y)=t\}}\frac{\vert\nabla u^{s}\vert^2}{\vert\partial_{y}u^{s}
\vert}\:d\HM^{0}(y)
=
\frac{2}{\vert\partial_{y}u^{s}\vert}\left(\sum_{i=1}^{d-1}\vert\partial_{i} u^{s}\vert^2+
\vert\partial_{y}u^{s}\vert^2\right)\Big\vert_{\{y\in\iphil:u^{s}(x',y)=t\}}\\
&&=
-\partial_{t}\mu_{u}(x',t)\left(\sum_{i=1}^{d-1}\frac{\vert\partial_{i}\mu_{u}(x',t)\vert^2}
{\vert\partial_{t}\mu_{u}(x',t)\vert^2}+\frac{4}{\vert\partial_{t}\mu_{u}(x',t)\vert^2}\right)
\Big\vert_{\{y\in\iphil:u^{s}(x',y)=t\}}\\
&&=
\int\limits_{\{y\in\iphil:u(x',y)=t\}}\frac{1}{\vert\partial_{y}u\vert}\:d\HM^{0}(y)
\left(\sum_{i=1}^{d-1}\frac{\left(\int\limits_{\{y\in\iphil:u(x',y)=t\}}\frac{\partial_{i}u}
{\vert\partial_{y}u\vert}\:d\HM^{0}(y)\right)^2}{\left(\int\limits_{\{y\in\iphil:u(x',y)=t\}}\frac{1}
{\vert\partial_{y}u\vert}\:d\HM^{0}(y)\right)^2}
+
\frac{4}{\left(\int\limits_{\{y\in\iphil:u(x',y)=t\}}\frac{1}{\vert\partial_{y}u\vert}\:d\HM^{0}(y)
\right)^2}\right).
\end{eqnarray*}
According to the Cauchy-Schwarz inequality, there holds
\begin{eqnarray}\label{cau1}
\lefteqn{\left(\int\limits_{\{y\in\iphil:u(x',y)=t\}}
\frac{\partial_{i}u}{\vert\partial_{y}u\vert}\:d\HM^{0}(y)\right)^2}\notag\\
&&\leq
\int\limits_{\{y\in\iphil:u(x',y)=t\}}\frac{\vert\partial_{i}u\vert^2}{\vert\partial_{y}u\vert}
\:d\HM^{0}(y)
\int\limits_{\{y\in\iphil:u(x',y)=t\}}\frac{1}{\vert\partial_{y}u\vert}\:d\HM^{0}(y),
\end{eqnarray}
with equality  if and only if $\partial_{i}u=c_{i}(x',t)$ for some function $c_{i}(x',t)$ that does not
depend on $y$. This implies
\begin{eqnarray}\label{appcf}
\lefteqn{\int\limits_{\{y\in\iphil:u^{s}(x',y)=t\}}\frac{\vert\nabla u^{s}\vert^2}{\vert\partial_{y}
u^{s}\vert}\:d\HM^{0}(y)}\notag\\
&&\leq
\sum_{i=1}^{d-1}\int\limits_{\{y\in\iphil:u(x',y)=t\}}\frac{\vert\partial_{i}u\vert^2}
{\vert\partial_{y}u\vert}\:d\HM^{0}(y)
+
\frac{4}{\int\limits_{\{y\in\iphil:u(x',y)=t\}}\frac{1}{\vert\partial_{y}u\vert}\:d\HM^{0}(y)}.
\end{eqnarray}
Using that $u$ is $\ell$-periodic, we deduce for all $t\in (m(x'),M(x'))$
from the isoperimetric inequality on $S^1$  that
\begin{eqnarray}\label{appiso}
2\leq\HM^{0}(\{y\in\iphil:u(x',y)=t\})=\int\limits_{\{y\in\iphil:u(x',y)=t\}}\:d\HM^{0}(y).
\end{eqnarray}
Thus we may estimate
\begin{eqnarray}\label{cau2}
4&\overset{\eqref{appiso}}\leq&\left(\int\limits_{\{y\in\iphil:u(x',y)=t\}}\:d\HM^{0}(y)
\right)^2\notag\\
&\leq&
\int\limits_{\{y\in\iphil:u(x',y)=t\}}\frac{\vert\partial_{y}u\vert^2}{\vert\partial_{y}u\vert}
\:d\HM^{0}(y)
\int\limits_{\{y\in\iphil:u(x',y)=t\}}\frac{1}{\vert\partial_{y}u\vert}\:d\HM^{0}(y),
\end{eqnarray}
where for the second inequality we have again used the Cauchy-Schwarz inequality. In this
case equality holds if and only if $\vert\partial_{y}u\vert=c(x',t)$ for some nonnegative function
$c(x',t)$ that does not depend on $y$. Substituting \eqref{cau2} into \eqref{appcf} yields
\begin{eqnarray}\label{cau3}
\lefteqn{\int\limits_{\{y\in\iphil:u^{s}(x',y)=t\}}\frac{\vert\nabla u^{s}\vert^2}{\vert\partial_{y}
u^{s}\vert}\:d\HM^{0}(y)}\notag\\
&\leq&
\sum_{i=1}^{d-1}\int\limits_{\{y\in\iphil:u(x',y)=t\}}\frac{\vert\partial_{i}u\vert^2}
{\vert\partial_{y}u\vert}\:d\HM^{0}(y)
+
\int\limits_{\{y\in\iphil:u(x',y)=t\}}\frac{\vert\partial_{y}u\vert^2}{\vert\partial_{y}u\vert}
\:d\HM^{0}(y)\notag\\
&=&
\int\limits_{\{y\in\iphil:u(x',y)=t\}}\frac{\vert\nabla u\vert^2}{\vert\partial_{y}u\vert}
\:d\HM^{0}(y).
\end{eqnarray}
Integrating \eqref{cau3} with respect to $t$ and using the one dimensional
Coarea Formula again, we see that the condition \eqref{inteq} implies  equality in \eqref{cau3}
for almost all $x'\in\tphil'$ and hence in all four
inequalities \eqref{cau1}, \eqref{appcf}, \eqref{appiso} and \eqref{cau2}
for almost all $x'\in\tphil'$ and  $\M^1$-a.e. $t\in (m(x'),M(x'))$ (which because of \eqref{mM} is nonempty).

We now augment $N$ by a set of $\M^{d-1}$ measure zero so that equality holds in
\eqref{cau1}-\eqref{cau2} for all $x'\in\tphil'\setminus N$.

\underline{Step 3.}[Using Step 2 to deduce 'bump structure' of $u(x',\cdot)$ and define $b$.]
We begin by observing that equality in \eqref{appiso} for almost all $x'\in\tphil'$ and  $\M^1$-a.e.
$t\in (m(x'),M(x'))$ improves to
equality in \eqref{appiso} for all $x'\in\tphil'$ and all $t\in (m(x'),M(x'))$, using continuity of $u$ (and arguing
by contradiction, for instance).

We will now describe the structure of $\{y\in\iphil\colon u(x',y)>t\}$. For fixed $x'\in\tphil'$
and $t\in(m(x'),M(x'))$, equality in \eqref{appiso} implies that the set $\{y\in\iphil\colon u(x',y)>t\}$ is
equal to an
open interval or its complement.
In other words, defining
\begin{align}
y_1(x',t):=\sup &\left\{ y\in \iphil\colon u(x',y)=t\text{ and there exists }\alpha\in(-\ell,y)\notag\right.\\
&\left.\quad \text{ such that }u(\alpha)<t\text{ and }u\text{ is nondecreasing on }(\alpha,y)\right\}\label{y1}\\
y_2(x',t):=\inf & \left\{y\in\iphil\colon u(x',y)=t\text{ and there exists }\alpha\in (-\ell,y)\notag\right.\\
&\left.\quad \text{ such that }u(\alpha)>t\text{ and }u\text{ is nonincreasing on }
(\alpha,y)\right\},\label{y2}
\end{align}
we have that
\begin{equation}
\begin{split}
\{y\in\iphil: u(x',y)>t\}&=(y_1(x',t), y_2(x',t))\\
\text{ or }\qquad \{y\in\iphil: u(x',y)>t\}&=[-\ell, y_2(x',t))\cup (y_1(x',t),\ell).
\end{split}\label{cau4}
\end{equation}
In particular, up to an $x'$-dependent shift, the graph of $u(x',\cdot)$ has the form of a 'bump': It is nondecreasing on $(-\ell,\alpha)$ and nonincreasing on $(\alpha,\ell)$ for some $\alpha\in \iphil$.

Using the above definitions of $y_1$ and $y_2$, we define
\begin{eqnarray}\label{bdef}
b(x',t):=\frac{1}{2}\left(y_1(x',t)+y_2(x',t)\right)
\end{eqnarray}
and observe that $b(\cdot,t)$ is a measurable function for each $t\in (m(x'), M(x'))$.

\underline{Step 4.}[The function $b$ is independent of $t$.]
We first consider $x'\in\tphil'\setminus N$.
As observed above,  equality  in \eqref{cau1} and \eqref{cau2} implies the existence of functions $c_{i}
(x',t)$ and $c(x',t)\geq 0$ (which do not depend on $y$), such that
\begin{eqnarray}
\label{cau5} \partial_{i}u(x',y_1(x',t)) &=&\partial_{i}u(x',y_2(x',t))=c_{i}(x',t),\qquad i=1,\ldots,d-1,
\\
\label{cau6} \vert\partial_{y}u(x',y_1(x',t))\vert &=&\vert\partial_{y}u(x',y_2(x',t))\vert=c(x',t)
\end{eqnarray}
for all $x'\in\tphil'\setminus N$ and for almost all $t\in (m(x'), M(x'))$.
In particular we have
\begin{eqnarray}\label{cau8}
\vert\nabla u(x',y_1)\vert=\vert\nabla u(x',y_2)\vert.
\end{eqnarray}
Additionally, using the definition of $y_j$ as the endpoints of the set $\{y\in\iphil: u(x',y)>t\}$, we improve from \eqref{cau6} to
\begin{eqnarray}\label{cau7}
\partial_{y}u(x',y_1(x',t))=-\partial_{y}u(x',y_2(x',t))
\end{eqnarray}
for all $x'\in\tphil'\setminus N$ and for almost all $t\in (m(x'), M(x'))$.

To fix ideas and simplify notation, we find it convenient to shift so that the first case in \eqref{cau4} holds. Hence let
\begin{align*}
&\alpha_1^+(x'):=\underset{t\uparrow M(x')}\lim y_1(x',t),\quad\alpha_2^+
(x'):=\underset{t\uparrow M(x')}\lim y_2(x',t),\\
&\alpha_1^-(x'):=\underset{t\downarrow m(x')}\lim y_1(x',t),\quad\alpha_2^-
(x'):=\underset{t\downarrow m(x')}\lim y_2(x',t)
\end{align*}
and consider the interval $\iphil\to(\alpha_1^-,\alpha_1^-+2\ell)$ so that $y_2>y_1$.
For $j=1,2$ we define the intervals
\begin{align*}
  I_j:=(\alpha_j^-,\alpha_j^+)
\end{align*}
and the corresponding distribution functions
\begin{eqnarray*}
\mu_j(x',t)&=\M^1\left(\left\{y\in I_j: t<u(x',t)<M(x')\right\}\right).
\end{eqnarray*}
Because \eqref{keyass3} holds on $\tphil'\setminus N$, these distribution functions can be written
 as
\begin{eqnarray*}
\mu_{j}(x',t)
&=&
\int\limits_{t}^{M(x')}\int\limits_{\{y\in I_{j}:u(x',y)=s\}}\frac{1}{\vert\partial_{y}u(x',y)\vert}
\:d\HM^0(y)\:ds\\
&=&
\int\limits_{t}^{M(x')}\frac{1}{\vert\partial_{y}u(x',y_{j}(x',s))\vert}\:ds\qquad\hbox{for}\:j=1,2,
\end{eqnarray*}
where we have for the second equality applied \eqref{cau4}.
Using this integral representation together with \eqref{cau6}, we conclude $\mu_{1}(x',t)=\mu_{2}(x',t)$ for all $t\in (m(x'),M(x'))$. Since
\begin{eqnarray*}
y_1(x',t)=\alpha_1^+(x')-\mu_1(x',t)\qquad\hbox{and}\qquad y_2(x',t)=\alpha_2^+(x')+\mu_2(x',t),
\end{eqnarray*}
we obtain as desired
\begin{eqnarray}
b(x',t)\overset{\eqref{bdef}}{=}\frac{y_1(x',t)+y_2(x',t)}{2}=\frac{\alpha_1(x')+\alpha_2(x')}
{2}\qquad\text{for all }t\in (m(x'),M(x')).\label{b*}
\end{eqnarray}

We now consider $x_0'\in N$.
The theorem of Sard in one dimension implies  for $\M^1$-a.e. $t\in (m(x_0'),M(x_0'))$ that
\begin{eqnarray*}
\partial_{y}u(x'_0,y_{j}(x'_0,t))\ne 0.
\end{eqnarray*}
Hence for any such $t$, we deduce from the Implicit Function Theorem that there exists an open set $U'=U'(x_0')$ such that the functions $y_j$ are $C^1$ in $x'$ on $U'$.
But then for any such $t$ and any sequence $x_n'\to x_0'$ with $x_n'\in \tphil'\setminus N$, there holds
\begin{align*}
b(x_0',t)=  \frac{y_1(x_0',t)+y_2(x_0',t)}{2}=\lim_{n\to\infty}\frac{y_1(x_n',t)+y_2(x_n',t)}{2}.
\end{align*}
Since the right-hand side is, according to \eqref{b*}, independent of $t$, so too is the left-hand side.
Finally using continuity of $u(x_0',\cdot)$ and the definitions of $y_1$ and $y_2$, we deduce from this equality for almost all $t$ that in fact $b(x_0',\cdot)$ is constant for all $t\in (m(x_0'),M(x_0'))$.

\underline{Step 5.}[The function $b$ is in $W^{1,1}$.]
We now establish $W^{1,1}$ regularity of $b$. Fix any $x_0'\in \tphil'$ and any $t_0\in(m(x_0'),M(x_0'))$ such that
\begin{eqnarray}
\partial_{y}u(x'_0,y_{j}(x'_0,t_0))\ne 0,\qquad j=1,\,2.\label{nec}
\end{eqnarray}
By shifting as in Step 2, we may without loss of generality assume that
\begin{align*}
y_1(x_0',t_0)< y_2(x_0',t_0).
\end{align*}
Moreover, continuity of $u$ and the $y_j$ (from the Implicit Function Theorem, as above) implies that $t_0\in (m(x'),M(x'))$ and this \emph{single shift} delivers
\begin{align}
y_1(x',t_0)<y_2(x',t_0)\label{y1less}
\end{align}
for all $x'$ in
a neighborhood $U'$ of $x_0'$.

Because of \eqref{y1less}, we have
the representation
\begin{eqnarray*}
b(x')=\frac{\int\limits_{\{y\in \iphil:u(x',y)>t_0\}}y\:dy}{\mu_{u}(x',t_0)}
\end{eqnarray*}
and $\mu(x',t_0)=y_2(x',t_0)-y_1(x',t_0)$. By choosing a $d-1$-dimensional ball $B_\rho(x_0')\Subset U'$ we may moreover assume
$\mu(x',t_0)\geq c$ on $B_\rho(x_0')$ for some $c>0$. Let
\begin{eqnarray*}
h(x'):=\int\limits_{\{y\in \iphil:u(x',y)>t_0\}}y\:dy.
\end{eqnarray*}
We will show $h\in W^{1,1}(B_\rho(x_0'))$. Let $\varphi\in C^1_{0}(B_\rho(x_0'))$.
Then with the same computations leading to \cite[formula (4.46)]{CF} and the additional
information (from the Implicit Function Theorem) that $\partial \{(x',y)\in\tphil\colon x'\in B_{\rho}(x_0')\;\text{and}\; u(x',y)>t_0\}$ is smooth, we get for $i=1,\hdots,d-1$ that
\begin{eqnarray*}
\int\limits_{B_{\rho}(x_0')}\varphi(x')\: d(\partial_{i}h(x'))
&=&
-\int\limits_{B_{\rho}(x_0')}\partial_{i}\varphi(x')\int\limits_{\iphil}\chi_{\{y\in \iphil:u(x',y)>t_0\}}(y)\:y\:dy\:dx'\\
&=&
-\int\limits_{B_{\rho}(x'_0)}\int\limits_{y_1(x',t_0)}^{y_2(x',t_0)}\partial_{i}\varphi(x')\:y\:dy\:dx'\\
&=&
-\int\limits_{B_{\rho}(x'_0)}\varphi(x')\:y\left(\partial_{i}y_2(x',t_0)-\partial_{i}y_1(x',t_0)\right)\:dx'.
\end{eqnarray*}
Then we conclude as in \cite{CF} (see Lemma 4.1 and Lemma 4.10): For any Borel set $B\Subset B_\rho(x_0')$ with $\M^{d-1}(B)=0$, there holds
\begin{eqnarray*}
\vert\partial_{i}h\vert(B)
\leq
\int\limits_{B_{\rho}(x'_0)\cap B}\:\vert y\vert\vert\partial_{i}y_2(x',t_0)-\partial_{i}y_1(x',t_0)\vert\:dx'=0.
\end{eqnarray*}
Thus $\partial_{i}h$ is absolutely continuous with respect to the $d-1$ dimensional Lebesgue
measure and $h\in W^{1,1}(B_\rho(x_0'))$.
A covering argument then gives $h\in W^{1,1}(\tphil')$. The analogous argument gives
$\mu_{u}(\cdot,t_0)\in W^{1,1}(\tphil')$ and
hence $b\in W^{1,1}(\tphil')$.

\underline{Step 6.}[The function $b$ does not depend on $x'$.]
We finally show that $b$ is constant. According to the previous step it suffices to show that $\partial_i b(x_0')=0$ for all $x_0'\in\tphil'\setminus N$ and all $i=1,\ldots,d$.

Fix any $x_0'\in\tphil'\setminus N$ and
$t_0\in (m(x_0'),M(x_0'))$ such that \eqref{nec} holds.

Similarly to in the previous step, we restrict to a small ball $B_\rho(x_0')$ such that $t_0\in (m(x'),M(x'))$ for all $x'\in B_\rho(x_0')$ so that $y_1(x',t_0)$, $y_2(x',t_0)$ are well-defined and $C^1$ on $B_\rho(x_0')$. For reference below we record the identity
there holds
\begin{eqnarray}\label{ift}
u(x',y_{j}(x',t_0))=t_0 \qquad\hbox{for}\:x'\in B_{\rho}(x'_0),\quad j=1,2.
\end{eqnarray}
We remark in addition that smoothness of $u$ and $y_j$ implies that the relations \eqref{cau5} and \eqref{cau7} hold for all $x'\in B_\rho(x_0')$ and in particular for $t=t_0$.
Furthermore we decrease $\rho>0$ if necessary so that
\begin{eqnarray}
\partial_{y}u(x',y_{j}(x',t_0))\ne 0,\qquad x'\in B_\rho(x_0'),\quad j=1,\,2.\label{nec2}
\end{eqnarray}

First we deduce from \eqref{ift} for $j=1,\, 2$ that
\begin{eqnarray*}
0=\frac{d}{dx'_{i}}u(x',y_{j}(x',t_0))
=
\partial_{i}u(x',y_{j}(x',t_0))+\partial_{y}u(x',y_{j}(x',t_0))\:\partial_{i}y_{j}(x',t_0).
\end{eqnarray*}
Using this together with \eqref{cau5} and \eqref{cau7} for $x'\in B_\rho(x_0')$ and $t=t_0$, we obtain
\begin{eqnarray*}
\partial_{y}u(x',y_{1}(x',t_0))\left(\partial_{i}y_{1}(x',t_0)+\partial_{i}y_{2}(x',t_0)\right)=0
\end{eqnarray*}
Recalling~\eqref{nec2}, we deduce
\begin{eqnarray*}
0=\partial_{i}y_{1}(x',t_0)+\partial_{i}y_{2}(x',t_0)=2\partial_{i}b(x')
\end{eqnarray*}
as desired.
\end{proof}
\section{Steiner Symmetrization applied to the Cahn-Hilliard problem}\label{S:steinerCH}
We now use Theorem~\ref{steinerthm} to give a first proof of Theorem~\ref{t:ssym} for minimizers of $\E$ over $\X$. We note for reference below that such minimizers are smooth and satisfy the Euler Lagrange equation
\begin{align}\label{el}
-\Delta u+f(u)=0\qquad\hbox{in}\:\tphil,
\end{align}
where
\begin{align}\label{elf}
f(u)=\frac{1}{\phi^2}G'(u)+\frac{1}{\phi}\left(\lambda_{\phi}+\lambda_{\omega}\zeta'(u)\right)
\end{align}
and $\lambda_\phi,\,\lambda_\omega$ are Lagrange parameters corresponding to the constraints.
Consequently for any minimizer $u\in \X$ and index $i\in\{1,\ldots ,d\}$, the partial derivative $\partial_{i}u$
satisfies the linear equation
\begin{align}
-\Delta(\partial_{i}u)+f'(u)(\partial_{i}u)=0\qquad\hbox{in}\quad\tphil.\label{ieq}
\end{align}

We will also utilize the following Strong Maximum Principle, due to Serrin \cite[Theorem 2.10]{HL}.
\begin{theorem}\label{t:strongmax}
Let $\Omega\subseteq\R$ be an open, bounded domain.
Suppose $u \in C^2(\Omega) \cap C(\bar{\Omega})$ satisfies $-\Delta u+c(x)u = 0$ in $\Omega$, where $c\in C(\bar{\Omega})$. If $u \leq 0$ in $\Omega$, then either
$u< 0$ in $\Omega$ or $u \equiv 0$ in $\bar{\Omega}$.
\end{theorem}
\subsection{Proof of Theorem \ref{t:ssym} via Steiner symmetrization}
We will now show that Hypothesis \ref{hyp:1} is satisfied by constrained minimizers
$u\in X_{\phi,\omega}$  of the Cahn-Hilliard energy $\E$. We denote by $u^s$
the Steiner symmetrized solution with respect to the $d$ - coordinate and set $y=x_d$.
By  \eqref{rearr2} we have $u^{s}\in\X$ as well. Moreover \eqref{rearr1} and \eqref{rearr2} give
\begin{eqnarray*}
\E(u^s)\leq \E(u).
\end{eqnarray*}
Thus $u^s$ is also a constrained minimizer of $\E$ and satisfies
 \eqref{el}-\eqref{elf} (possibly for different Lagrange parameters $\lambda_{\phi}^s$ and
$\lambda_{\omega}^s$ than for $u$). Note that Remark \ref{rem:prop} gives
\begin{eqnarray*}
\partial_{y}u^s(x',y)\leq 0\quad\hbox{in}\quad\tphil'\times(0,\ell)\qquad\hbox{and}\qquad
\partial_{y}u^s(x',y)\geq 0\quad\hbox{in}\quad\tphil'\times(-\ell,0).
\end{eqnarray*} Clearly \eqref{ieq} holds for $\partial_{y}u^s$ as well. Consequently Theorem \ref{t:strongmax} gives
\begin{eqnarray*}
\partial_{y}u^s(x',y)< 0\quad\hbox{or}\quad \partial_{y}u^s(x',y)\equiv 0\quad\hbox{in}\quad\tphil'\times(0,\ell),
\end{eqnarray*}
with the analogous statement for $\tphil'\times(-\ell,0)$. The case of equality can be excluded, since in
Theorem 1.19 in \cite{GWW} it was shown that for $\phi$ small there exists a sharp-interface profile
$\Psi:\tphil\to\{-1,1\}$ such that $\Psi=+1$ in a ball and $\Psi=-1$ on the complement and such that  the volume constrained minimizer satisfies
\begin{eqnarray*}
\| u-\Psi\|_{L^2}\ll 1.
\end{eqnarray*}
Since Steiner symmetrization is nonexpansive (see e.g. \cite{K} Section II.2) we also get
\begin{eqnarray*}
\| u^s-\Psi\|_{L^2}\leq\| u-\Psi\|_{L^2}\ll 1.
\end{eqnarray*}
This rules out the case $\partial_{y}u^s(x',y)\equiv 0$. Consequently
\begin{eqnarray*}
m^s(x'):=\min\{u^s(x',y):y\in\iphil\}<\max\{u^s(x',y):y\in\iphil\}=:M^s(x').
\end{eqnarray*}
Since rearrangements preserve the maximum and minimum of a function, this implies
$m(x')<M(x')$ and thus \eqref{mM} holds. Since $u^s$ is strictly decreasing on $(0,\ell)$
(and increasing on $(-\ell,0)$) Lemma \ref{equimeas} implies \eqref{keyass3} and thus
also \eqref{keyass1}. Hence Hypothesis \ref{hyp:1} is satisfied.

Consequently we can immediately deduce Theorem \ref{t:ssym} from Theorem \ref{steinerthm}.
\subsection{Sphericity of constrained minimizers in $d=2$}\label{ss:bonn}
Using the connectedness of the superlevel sets from Theorem \ref{steinerthm} together with the 
Bonnesen inequality, we obtain a quantitative estimate on the sphericity of the superlevel sets of 
constrained minimizers in dimension $d=2$. Loosely speaking, we can show that for any constrained 
minimizer $\uo$, the superlevel sets $\{\uo >\eta\}$ for $\eta \in (-1,1)$  cannot possess ``tentacles'' and 
are therefore close to a ball in the sense of Hausdorff distance, improving the sphericity estimate in terms 
of the Fraenkel asymmetry from \cite{GWW}. Hence the possibility of mass drifting off to infinity as 
$\phi\downarrow 0$ is precluded. The main tool that is needed in order to establish this fact is the 
Bonnesen inequality, which we state below after recalling the definition of the outer and inner radius.
\begin{definition}
Consider a simply connected domain $A\subset\mathbb{R}^2$. The outer radius of $A$, denoted 
$\rho_{out}(A)$, is defined as the infimum of the radii of all the disks in $\mathbb{R}^2$ that contain $A$. 
Similarly, the inner radius of $A$, denoted $\rho_{in}(A)$, is defined as the supremum of the radii of all 
the disks in $\mathbb{R}^2$ that are contained in $A$. Lastly, we define the volume radius of $A$, 
denoted $\rho(A)$, as the radius of a disk in $\mathbb{R}^2$ whose measure is equal to that of $A$.
\end{definition}
\begin{remark}
We may use the same definition for the inner and outer radius of a simply connected domain $A\subset\tphil$, provided that there exists a disk in $\tphil$ that contains $A$. Note that in that case there holds $\rho_{out}(A)<(\phi L)/2$ and $\per_{\tphil}(A)=\per_{\mathbb{R}^2}(A)$.
\end{remark}
The classical Bonnesen inequality in the plane is as follows.
\begin{theorem}[Bonnesen inequality in $\mathbb{R}^2$]\label{Bonnesen.Ineq}
For any simply connected domain $A\subset\mathbb{R}^2$ with smooth boundary, there holds
\begin{equation}\label{Bonnesen.Ineq2}
\per_{\mathbb{R}^2}(A)\geq \sqrt{\pi}\left(4|A|+\Big(\rho_{out}(A)-\rho_{in}(A)\Big)^2 \right)^{1/2}.
\end{equation}
\end{theorem}
An application of the Bonnesen inequality to our problem yields the following result for constrained minimizers in the parameter regime \eqref{regime} for $\xi \in (\tilde{\xi_2},\xi_2]$, where the endpoints are given by
\begin{align}\label{xi.sub.d}
\tilde{\xi}_2:=\frac{3(c_0^2\pi)^{1/3}}{2^{5/3}},\quad \xi_2:=\sqrt{2}\tilde{\xi}_2
\end{align}
and
\begin{align*}
  c_0=\int_{-1}^{1}\sqrt{2G(s)}\,ds
\end{align*}
(which is $2\sqrt{2}/3$ for the standard potential $G(s)=(1-s^2)^2/4$). We refer to \cite{GWW} for the derivation and significance of $\tilde{\xi}_2$ and $\xi_2$.
\begin{proposition}\label{Asym.stronger}
Consider the critical regime \eqref{regime} with $\xi \in (\tilde{\xi_2},\xi_2]$.
Fix any $\omega_1>0$ and consider $\phi>0$ sufficiently small. For any volume-constrained minimizer $\uo$ with volume $\omega\in[\omega_1,\xi^3/2]$ and any $\eta\in(-1,1)$, there holds
\begin{equation}\label{str.0}
\rho_{out}(\{\uo >\eta \})=r_\omega+\frac{O(\phi^{1/6})}{(1+\eta)},
\end{equation}
and
\begin{equation}\label{str.0.1}
\rho_{in}(\{\uo >\eta \})=r_\omega+\frac{O(\phi^{1/6})}{(1-\eta)},
\end{equation}
where
\begin{equation}\label{str.0.2}
r_\omega:=\sqrt{\frac{\omega}{\pi}}.
\end{equation}
Consequently, there holds
\begin{equation}\label{str.0.3}
\rho_{out}(\{\uo >\eta \})-\rho_{in}(\{\uo >\eta \})=\frac{O(\phi^{1/6})}{(1-\eta^2)}.
\end{equation}
\end{proposition}

\begin{proof}[Proof of proposition \ref{Asym.stronger}]
According to Theorem~\ref{steinerthm}, $\uo$ is smooth and equal to its Steiner symmetrization. We also recall two facts from \cite{GWW}:
\begin{align}
&\int_{-1+\phi^{1/3}}^{1-\phi^{1/3}}\sqrt{2\tilde{G}(t)}\,\Big({\rm Per}_{\tphil}(\{\uo>t\})-{\rm P}_E(\{\uo>t\})\Big)\,dt\lesssim \phi^{1/3},\label{JJ}\\
\text{and }\quad&|\{u>t\}|=\omega+O(\phi^{1/3})\;\;\text{for all }t\in [-1+\phi^{1/3},1-\phi^{1/3}],\label{size}
\end{align}
where in \eqref{JJ} ${\rm Per}_{\tphil}$ represents the perimeter and ${\rm P}_E$ represents the perimeter of a ball with the same volume.
Next we claim that we may shift $\uo$ so that $\{\uo>-1+2\phi^{1/3}\}$ is contained within a disk centered at the origin and of radius less than $(\phi L)/2$. Indeed, if this is not the case, it follows by the Steiner symmetry of $\{\uo>-1+2\phi^{1/3}\}$ that

\[
\per_{\tphil}(\{\uo>-1+2\phi^{1/3}\})\geq\frac{1}{2}\phi L,
\]
which, because of the monotonicity of $\{\uo>t\}$ with respect to $t$, implies in turn that
\[
\per_{\tphil}(\{\uo>t\})\geq\frac{1}{2}\phi L \quad \text{for all}\quad t\in[-1+\phi^{1/3},-1+2\phi^{1/3}].
\]
It follows that
\[
\int_{-1+\phi^{1/3}}^{-1+2\phi^{1/3}}\sqrt{2\tilde{G}(t)}\,\Big({\rm Per}_{\tphil}(\{u>t\})-{\rm P}_E(\{u>t\})\Big)\,dt\gtrsim \phi^{1/6},
\]
which contradicts \eqref{JJ}.

We now establish a lower bound on $I(\uo)$. For any $t\in[-1+2\phi^{1/3},1-2\phi^{1/3}]$ we will denote by $\rho(t),\,\rho_{in}(t)$ and $\rho_{out}(t)$ the volume-, inner- and outer radius of $\{\uo>t\}$, respectively. We will also use the notation
\[
\mathit{\Delta} \rho(t):=\rho_{out}(t)-\rho_{in}(t).
\]
Because the superlevel sets of $\uo$ are contained within a disk (as discussed above), we may apply the Bonnesen inequality \eqref{Bonnesen.Ineq2} to $I$ to obtain
\begin{eqnarray*}
\phi^{1/3}&\overset{\eqref{JJ},\eqref{Bonnesen.Ineq2}}\gtrsim& \int_{-1+2\phi^{1/3}}^{1-2\phi^{1/3}}\sqrt{2\tilde{G}(t)}\,\left[(4\pi|\{\uo>t\}|+\pi(\mathit{\mathit{\Delta}} \rho(t))^2)^{1/2}-{\rm P}_E(\{\uo>t\})\right]\,dt \notag \\
&=&2\sqrt{2\pi}\int_{-1+2\phi^{1/3}}^{1-2\phi^{1/3}}\sqrt{\tilde{G}(t)}\,|\{\uo>t\}|^{1/2}\left[\left(1+\frac{(\mathit{\Delta} \rho(t))^2}{4|\{\uo>t\}|}\right)^{1/2}-1\right]\,dt\notag\\
&\gtrsim &\int_{-1+2\phi^{1/3}}^{1-2\phi^{1/3}}\sqrt{\tilde{G}(t)}\,\frac{(\mathit{\Delta} \rho(t))^2}{|\{\uo>t\}|^{1/2}}\,dt.
\end{eqnarray*}
We combine this with \eqref{size} to deduce
\begin{align}\label{str.1}
\int_{-1+2\phi^{1/3}}^{1-2\phi^{1/3}}\sqrt{\tilde{G}(t)}\,(\mathit{\Delta} \rho(t))^2\,dt\lesssim \phi^{1/3}.
\end{align}
Next we observe that, due to the monotonicity of the superlevel sets $\{\uo>t\}$ with respect to $t$, we have
\begin{equation}\label{str.2}
\rho_{out}(t)\geq\rho_{out}(\eta)\quad \text{for every}\quad t\in[-1+2\phi^{1/3},\eta],
\end{equation}
and
\begin{equation}\label{str.3}
\rho_{in}(t)\leq\rho_{in}(\eta)\quad \text{for every}\quad t\in[\eta,1-2\phi^{1/3}].
\end{equation}
Moreover, again due to monotonicity, for all $t\in[-1+2\phi^{1/3},1-2\phi^{1/3}]$ there holds
\begin{eqnarray}
\rho_{out}(t)&\geq&\rho_{out}(1-2\phi^{1/3})\geq \left(|\{\uo>1-2\phi^{1/3}\}|/\pi\right)^{1/2}\notag \\
&\overset{\eqref{size}}=&r_\omega+O(\phi^{1/6}),\label{str.4}
\end{eqnarray}
and
\begin{eqnarray}
\rho_{in}(t)&\leq&\rho_{in}(-1+2\phi^{1/3})\leq \left(|\{\uo>-1+2\phi^{1/3}\}|/\pi\right)^{1/2}\notag\\
&\overset{\eqref{size}}=&r_{\omega}+O(\phi^{1/6}).\label{str.5}
\end{eqnarray}

By \eqref{str.2} and \eqref{str.5} it follows that for all $t\in [-1+2\phi^{1/3},\eta]$ the difference $\mathit{\Delta} \rho(t)$ satisfies
\[
\mathit{\Delta} \rho(t)\geq  \rho_{out}(\eta)-r_\omega+O(\phi^{1/6}).
\]
Substituting into \eqref{str.1} implies \eqref{str.0}.

For  $t\in [\eta,1-2\phi^{1/3}]$ on the other hand, \eqref{str.3} and \eqref{str.4} imply that
\[
\mathit{\Delta} \rho(t)\geq  r_\omega-\rho_{in}(\eta)+O(\phi^{1/6}),
\]
which together with \eqref{str.1} yields \eqref{str.0.1}.

\end{proof}
\section{Two Point Rearrangement applied to the Cahn-Hilliard problem}\label{S:GNN}
The two-point rearrangement was first introduced in \cite{Bae} and extensively discussed in \cite{BS}.
The periodic variant was given in \cite{Br}.
For any $\eta\in I_\ell$ we define the reflection in the $y$-direction with reflection point $\eta$ as
\begin{eqnarray*}
x^{\eta}:=(x',y^{\eta}):=(x',2\eta-y)
\end{eqnarray*}
and
\begin{eqnarray*}
u^{\eta}(x):=u(x^{\eta}).
\end{eqnarray*}
Note that if $x\in [\eta,\eta+\ell]\times\tphil'$, then $x^{\eta}\in [\eta-\ell,\eta]\times\tphil'$.
It is convenient to state our main result using two-point rearrangements in the following form.
\begin{theorem}\label{t:twopointsymm}
Let $u$  minimize $\E$ over $\X$. Assume that \eqref{supp} holds. Then there exists
$\eta^\star\in I_\ell$
such that
\begin{eqnarray*}
u=u^{\eta^{\star}}\quad\text{on }\tphil.
\end{eqnarray*}
Moreover, there holds
\begin{eqnarray*}
\partial_{y}u(x',y)< 0 \quad\text{on }(\eta^{\star},\eta^{\star}+\ell)\times \tphil'.
\end{eqnarray*}
\end{theorem}
\begin{remark}
  Clearly Theorem~\ref{t:twopointsymm} provides an alternate proof of Theorem~\ref{t:ssym}.
\end{remark}
We will establish this result by way of the so-called two-point rearrangement.
\begin{definition}\label{tpr}
The \uline{two-point rearrangement} of a function $u\in W^{1,2}(\tphil)$ for $\eta\in[-\ell,\ell]$ is defined as
\begin{eqnarray*}
T^{\eta}u(x):=
\begin{cases}
\max\{u(x),u(x^{\eta})\}&\hbox{for all}\; x\in \tphil'\times[\eta,\eta+\ell]\\
\min\{u(x),u(x^{\eta})\}&\hbox{for all}\; x\in \tphil'\times[\eta-\ell,\eta],
\end{cases}
\end{eqnarray*}
and we will identify it with its $2\ell$-periodic continuation to the rest of $\rz^d$ and in particular to $\tphil$.
\end{definition}
Along the way, we will make use of the following Weak Unique Continuation Principle (cf.\ \cite[Section 1, case (I)]{Wo}).
\begin{theorem}\label{ucp}
Let $\Omega\subseteq\R$ be a bounded domain and $c\in L^{\infty}(\Omega)$. Let $u\in W^{2,2}(\Omega)$ satisfy
\begin{eqnarray*}
-\Delta u+c(x)u=0\qquad\hbox{in}\:\Omega.
\end{eqnarray*}
Assume there is a nonempty, open subset $U\subseteq\Omega$ such that $u\equiv 0$ in $U$. Then $u\equiv 0$ in $\Omega$.
\end{theorem}
\subsection{Background and a rigidity result}
In this subsection we develop the necessary background that we need to prove Theorem~\ref{t:twopointsymm}.
We start by collecting a few elementary properties of the two-point rearrangement.
\begin{remark}\label{prop1}
The following statements are equivalent:
\begin{itemize}
\item[(i)] $u(x)=T^{\eta}u(x)$ in $\tphil$;
\item[(ii)] $u(x)\geq u(x^{\eta})$ for all $x\in \tphil'\times[\eta,\eta+\ell]$;
\item[(iii)] $u(x)\leq u(x^{\eta})$ for all $x\in \tphil'\times[\eta-\ell,\eta]$.
\end{itemize}
\end{remark}
\begin{lemma}\label{prop2}
 If $u(x)=T^{\eta}u(x)$ in $\tphil$, then
\begin{eqnarray*}
T^{\eta+\ell}u(x)=u^{\eta}(x)\qquad\hbox{for all}\:x\in\tphil'\times[\eta-\ell,\eta+\ell].
\end{eqnarray*}
Analogously, if $u^{\eta}(x)=T^{\eta}u(x)$ in $\tphil$, then
\begin{eqnarray*}
T^{\eta+\ell}u(x)=u(x)\qquad\hbox{for all}\; x\in\tphil'\times[\eta-\ell,\eta+\ell].
\end{eqnarray*}
\end{lemma}
\begin{proof}
It suffices to prove the first statement, since the first statement implies the second. From the definition of $T^{\eta}$, we observe
\begin{eqnarray*}
T^{\eta+\ell}u(x)
&=&
\begin{cases}
\max\{u(x),u(x^{\eta+\ell})\}&\hbox{for all}\; x\in\tphil'\times[\eta+\ell,\eta+2\ell],\\
\min\{u(x),u(x^{\eta+\ell})\}&\hbox{for all}\; x\in\tphil'\times[\eta,\eta+\ell].
\end{cases}
\end{eqnarray*}
Since $u$ is $2\ell$-periodic in the $y$-variable, there holds
\begin{eqnarray*}
u(x^{\eta+\ell})=u(x',2(\eta+\ell)-y)=u(x',2\eta-y+2\ell)=u(x',2\eta-y)=u(x^{\eta}).
\end{eqnarray*}
We use this to conclude
\begin{eqnarray*}
T^{\eta+\ell}u(x)
&=&
\begin{cases}
\max\{u(x),u(x^{\eta})\}&\hbox{for all}\; x\in\tphil'\times[\eta+\ell,\eta+2\ell]\\
\min\{u(x),u(x^{\eta})\}&\hbox{for all}\; x\in\tphil'\times[\eta,\eta+\ell],
\end{cases}\\
&=&
\begin{cases}
\max\{u(x),u(x^{\eta})\}&\hbox{for all}\; x\in\tphil'\times[\eta-\ell,\eta]\\
\min\{u(x),u(x^{\eta})\}&\hbox{for all}\; x\in\tphil'\times[\eta,\eta+\ell],
\end{cases}
\end{eqnarray*}
where for the second equality we again  used the $2\ell$-periodicity of $u$.
Remark \ref{prop1} then yields the result.
\end{proof}
The following  lemma is an adaptation of \cite{BS} to the periodic setting.
The first lemma implies that if $u\in X$ is a minimizer of $\E$, then  $T^{\eta}u\in X$ is also a minimizer of $\E$.
\begin{lemma}\label{proposition}
Let $u\in W^{1,2}(\tphil)$ and $\eta\in[-\ell,\ell]$. Then $T^{\eta}u\in W^{1,2}(\tphil)$ and
\begin{eqnarray}\label{prop3}
\int\limits_{\tphil}\vert Du\vert^2\:dx=\int\limits_{\tphil}\vert DT^{\eta}u\vert^2\:dx.
\end{eqnarray}
If $u\in C(\tphil)$, then $T^{\eta}u\in C(\tphil)$ for any $\eta\in[-\ell,\ell]$ and for any $G\in C^1(\tphil)$ there holds
\begin{eqnarray}\label{prop4}
\int\limits_{-\ell}^{\ell}G(u(x',y))\:dy =\int\limits_{-\ell}^{\ell}G(T^{\eta}u(x',y))\:dy\qquad\hbox{for all}\:x'\in\tphil'.
\end{eqnarray}
\end{lemma}
\begin{proof}
We give the proof of \eqref{prop3}; the proof of \eqref{prop4} is similar.
It is enough to consider the integral with respect to $y$, which we decompose as
\begin{eqnarray*}
\int\limits_{-\ell}^{\ell}\vert DT^{\eta}u\vert^2\:dy
=
\int\limits_{-\ell}^{\eta-\ell}\vert DT^{\eta}u\vert^2\:dy
+
\int\limits_{\eta-\ell}^{\eta}\vert DT^{\eta}u\vert^2\:dy
+
\int\limits_{\eta}^{\ell}\vert DT^{\eta}u\vert^2\:dy.
\end{eqnarray*}
The $2\ell$-periodicity of $T^{\eta}u$ implies
\begin{eqnarray*}
\int\limits_{-\ell}^{\eta-\ell}\vert DT^{\eta}u\vert^2\:dy=\int\limits_{\ell}^{\eta+\ell}\vert DT^{\eta}u\vert^2\:dy,
\end{eqnarray*}
so that
\begin{eqnarray*}
\int\limits_{-\ell}^{\ell}\vert DT^{\eta}u\vert^2\:dy
&=&
\int\limits_{\eta-\ell}^{\eta}\vert DT^{\eta}u\vert^2\:dy
+
\int\limits_{\eta}^{\eta+\ell}\vert DT^{\eta}u\vert^2\:dy.
\end{eqnarray*}
We use the definition of $T^{\eta}u$ and split the domain of integration as
\begin{eqnarray*}
&&\int\limits_{-\ell}^{\ell}\vert DT^{\eta}u\vert^2\:dy\\
&=&
\int\limits_{(\eta-\ell,\eta)\cap\{y:u(x',y)\geq u^{\eta}(x',y)\}}\vert Du^{\eta}\vert^2\:dy
+
\int\limits_{(\eta-\ell,\eta)\cap\{y:u(x',y)< u^{\eta}(x',y)\}}\vert Du\vert^2\:dy\\
&&+
\int\limits_{(\eta,\eta+\ell)\cap\{y:u(x',y)> u^{\eta}(x',y)\}}\vert Du\vert^2\:dy
+
\int\limits_{(\eta,\eta+\ell)\cap\{y:u(x',y)\leq u^{\eta}(x',y)\}}\vert Du^{\eta}\vert^2\:dy.
\end{eqnarray*}
The change of variable $\tilde{y}=2\eta-y$ in the first and fourth integrals and the identities
\begin{eqnarray*}
u(x',y)=u(x',2\eta-\tilde{y})=u^{\eta}(x',\tilde{y}),
\qquad\hbox{and}\qquad u^{\eta}(x',y)=u(x',2\eta-y)=u(x',\tilde{y})
\end{eqnarray*}
lead to
\begin{eqnarray*}
&&\int\limits_{-\ell}^{\ell}\vert DT^{\eta}u\vert^2\:dy\\
&=&
\int\limits_{(\eta,\eta+\ell)\cap\{y:u^{\eta}(x',\tilde{y})\geq u(x',\tilde{y})\}}\vert Du\vert^2\:d\tilde{y}
+
\int\limits_{(\eta-\ell,\eta)\cap\{y:u(x',y)< u^{\eta}(x',y)\}}\vert Du\vert^2\:dy\\
&&+
\int\limits_{(\eta,\eta+\ell)\cap\{y:u(x',y)> u^{\eta}(x',y)\}}\vert Du\vert^2\:dy
+
\int\limits_{(\eta-\ell,\eta)\cap\{y:u^{\eta}(x',\tilde{y})\leq u(x',\tilde{y})\}}\vert Du\vert^2\:d\tilde{y}\\
&=&
\int\limits_{-\ell}^{\ell}\vert Du\vert^2\:dy.
\end{eqnarray*}
\end{proof}
Next we prove a statement about the dependence of the Lagrange multipliers $\lambda_{\phi}$ and $\lambda_{\omega}$ from \eqref{el}, \eqref{elf} on solutions and their reflections.
\begin{lemma}\label{lagrange}
Let $u\in X$ be a minimizer of $\E$. For any $\eta\in[-\ell,\ell]$ the Lagrange multipliers $\lambda_{\phi}$ and $\lambda_{\omega}$ from \eqref{el}-\eqref{elf} satisfy
\begin{eqnarray*}
\lambda_{\phi}(u)=\lambda_{\phi}(u^{\eta})\qquad\hbox{and}\qquad\lambda_{\omega}(u)=\lambda_{\omega}(u^{\eta}).
\end{eqnarray*}
\end{lemma}
\begin{proof}
Clearly for any $\eta\in[-\ell,\ell]$ the function $u^{\eta}$ is also a minimizer of $\E$ and thus satisfies \eqref{el}--\eqref{elf} with Lagrange parameters $\lambda_{\phi}(u^{\eta}),\, \lambda_{\omega}(u^{\eta})$.
Multiplying \eqref{el} for $u$ and $u^{\eta}$ by $u-\overline{u}$ and $u^{\eta}-\overline{u^{\eta}}$, respectively, and integrating gives
\begin{align}
0&=\int\limits_{\tphil}\vert D  u\vert^2\:dx+\frac{1}{\phi^2}\int\limits_{\tphil}G'(u)(u-\overline{u})\:dx+\frac{\lambda_{\omega}(u)}{\phi}\int\limits_{\tphil}\zeta'(u)(u-\overline{u})\:dx,
\label{eq1}\\
0&=\int\limits_{\tphil}\vert D  u^{\eta}\vert^2\:dx+\frac{1}{\phi^2}\int\limits_{\tphil}G'(u^{\eta})(u^{\eta}-\overline{u^{\eta}})\:dx+
\frac{\lambda_{\omega}(u^{\eta})}{\phi}\int\limits_{\tphil}\zeta'(u^{\eta})(u^{\eta}-\overline{u^{\eta}})\:dx.\notag
\end{align}
The change of variables $y_1=2\eta-y$, $y'=x'$ in the second equation yields
\begin{align}
0=\int\limits_{\tphil}\vert D  u\vert^2\:dy+\frac{1}{\phi^2}\int\limits_{\tphil}G'(u)(u-\overline{u})\:dy+\frac{\lambda_{\omega}(u^{\eta})}{\phi}\int\limits_{\tphil}\zeta'(u)(u-\overline{u})\:dy.\label{eq2}
\end{align}
From \eqref{eq1} and \eqref{eq2} we deduce $\lambda_{\omega}(u)=\lambda_{\omega}(u^{\eta})$.

Integrating \eqref{el} for $u$ and $u^{\eta}$ gives
\begin{align*}
0&=\frac{1}{\phi}\int\limits_{\tphil}G'(u)\:dx+\lambda_{\phi}(u)+\lambda_{\omega}(u)\int\limits_{\tphil}\zeta'(u)\:dx,\\
0&=\frac{1}{\phi}\int\limits_{\tphil}G'(u^{\eta})\:dx+\lambda_{\phi}(u^{\eta})+\lambda_{\omega}(u^{\eta})\int\limits_{\tphil}\zeta'(u^{\eta})\:dx.
\end{align*}
Changing variables  in the second equation as above and using $\lambda_{\omega}(u)=\lambda_{\omega}(u^{\eta})$ yields $\lambda_{\phi}(u)=\lambda_{\phi}(u^{\eta})$.
\end{proof}
The following ``rigidity'' result provides the core of our argument. Using the equality of Lagrange parameters from the previous lemma, we are able to apply the Weak Unique Continuation Principle to conclude that one of two alternatives holds for each shift parameter $\eta$. The statement does not exclude that both alternatives may occur.
\begin{proposition}\label{either}
Let $u\in X$ be a minimizer of $\E$. For any $\eta\in[-\ell,\ell]$ we have
\begin{eqnarray*}
u=T^{\eta}u\qquad\hbox{or}\qquad u^{\eta}=T^{\eta}u\qquad\hbox{in}\quad\tphil.
\end{eqnarray*}
\end{proposition}
\begin{proof}
By Lemma \ref{proposition}
for any $\eta\in[-\ell,\ell]$ the function $T^{\eta} u$ is also a minimizer of $\E$ and hence satisfies \eqref{el}--\eqref{elf} with
 Lagrange parameters $\lambda_{\phi}(T^{\eta} u),\, \lambda_{\omega}(T^{\eta} u)$.

We will assume that $u\not\equiv T^{\eta}u$ and show that then $u^{\eta}=T^{\eta}u$. According to our assumption, the open set
\begin{eqnarray*}
U:=\left\{x\in \tphil'\times(\eta,\eta+\ell)\::\:u(x)<u(x^{\eta})\right\}
\end{eqnarray*}
is nonempty.

To begin, we will show that the Lagrange parameters of $u$ and $T^{\eta}u$ are equal. Note that the support of $\zeta'(u)$ is empty if and only if the support of $\zeta'(T^{\eta}u)$ is. In this case, the term with $\lambda_\omega$ in the Euler-Lagrange equation vanishes. Hence we may assume without loss that the support of $\zeta'(u)$ is nonempty.
By definition of $T^{\eta}u$ we have $u^{\eta}=T^{\eta}u$ in $U$ and $u=T^{\eta}u$ in $\tphil\setminus U$, so that
\begin{align*}
\lambda_{\phi}(u^{\eta})-\lambda_{\phi}(T^{\eta}u)+\big(\lambda_{\omega}(u^{\eta})-\lambda_{\omega}
(T^{\eta}u)\big)\zeta'(T^{\eta}u)&=0\qquad x\in U,\\
\lambda_{\phi}(u)-\lambda_{\phi}(T^{\eta}u)+\big(\lambda_{\omega}(u)-\lambda_{\omega}
(T^{\eta}u)\big)\zeta'(T^{\eta}u)&=0\qquad x\in \tphil\setminus U.
\end{align*}
Applying Lemma \ref{lagrange}, we obtain
\begin{eqnarray}
\lambda_{\phi}(u)-\lambda_{\phi}(T^{\eta}u)+\left(\lambda_{\omega}(u)-\lambda_{\omega}(T^{\eta}u)\right)\zeta'(T^{\eta}u)=0\qquad x\in \tphil.\label{l5.1}
\end{eqnarray}
According to \eqref{supp} there exist points outside the support of $\zeta'(T^{\eta}u)$, and at such points \eqref{l5.1} implies $\lambda_{\phi}(u)=\lambda_{\phi}(T^{\eta}u)$. But then \eqref{l5.1} at points within the support of $\zeta'(T^{\eta}u)$ yields $\lambda_{\omega}(u)=\lambda_{\omega}(T^{\eta}u)$. Combining this with Lemma \ref{lagrange} implies equality of all the Lagrange parameters:
\begin{align*}
\lambda_{\phi}(u)=\lambda_{\phi}(u^\eta)=\lambda_{\phi}(T^{\eta}u)\quad\text{and}\quad
  \lambda_{\omega}(u)=\lambda_{\omega}(u^\eta)=\lambda_{\omega}(T^{\eta}u).
\end{align*}

We now observe that
\begin{eqnarray*}
w:= u^{\eta}-T^{\eta}u
\end{eqnarray*}
satisfies
\begin{eqnarray*}
-\Delta w+c(x)w=0\qquad\hbox{in}\quad \tphil,
\end{eqnarray*}
where
\begin{eqnarray*}
c(x)=\int\limits_{0}^{1}f'\left(t u^{\eta}+(1-t)T^{\eta}u\right)\:dt.
\end{eqnarray*}
Since $w=0$ in $U$, the Weak Unique Continuation Principle (cf.\ Theorem~\ref{ucp}) implies $u^{\eta}=T^{\eta}u$ in $\tphil$.
\end{proof}
\subsection{Proof of Theorem~\ref{t:twopointsymm}}\label{SYM}
Using the background from the previous subsections, the proof of Theorem~\ref{t:twopointsymm} is straightforward.
\begin{proof}[Proof of Theorem~\ref{t:twopointsymm}]
We may without loss of generality assume that for
some $x'\in \tphil'$, there holds
$\partial_{y}u(x',y)\not\equiv 0$.

We recall from Proposition \ref{either} that for each $\eta\in[-\ell,\ell]$ there holds
\begin{eqnarray*}
(i)\quad u=T^{\eta}u\quad\hbox{in}\quad\tphil\qquad\hbox{or}\qquad (ii)\quad u^{\eta}=T^{\eta}u\quad\hbox{in}\quad\tphil.
\end{eqnarray*}

\medskip

\uline{Step 1.} We begin by showing that neither (i) nor (ii) can hold for all $\eta\in[-\ell,\ell]$. Assume without loss of generality that (i) holds for all $\eta\in[-\ell,\ell]$. W.l.o.g.\ consider $\eta\in[0,\ell]$. From the definition of $T^{\eta}u$ we obtain
\begin{eqnarray}\label{mu1}
\label{mu11} u(x)&\geq& u^{\eta}(x)\quad\hbox{for all}\; x\in \tphil'\times[\eta,\eta+\ell],\\
\label{mu12} u(x)&\leq& u^{\eta}(x)\quad\hbox{for all}\; x\in \tphil'\times[\eta-\ell,\eta].
\end{eqnarray}
On the other hand $\eta-\ell\in[-\ell,0]$ and hence \eqref{mu11} and \eqref{mu12} hold for parameter value $\eta-\ell$. Since $u^{\eta}=u^{\eta-\ell}$, we have
\begin{eqnarray*}
u(x)&\geq& u^{\eta}(x)\quad\hbox{for all}\; x\in \tphil'\times[\eta-\ell,\eta],\\
u(x)&\leq& u^{\eta}(x)\quad\hbox{for all}\; x\in \tphil'\times[\eta-2\ell,\eta-\ell].
\end{eqnarray*}
By the $2\ell$-periodicity of $u$ in the $y$-variable, this is equivalent to
\begin{eqnarray}
\label{mu21} u(x)&\geq& u^{\eta}(x)\quad\hbox{for all}\; x\in \tphil'\times[\eta-\ell,\eta],\\
\label{mu22} u(x)&\leq& u^{\eta}(x)\quad\hbox{for all}\; x\in \tphil'\times[\eta,\eta+\ell].
\end{eqnarray}
A comparison of \eqref{mu11} with \eqref{mu22} and \eqref{mu12} with \eqref{mu21} gives $u=u^{\eta}$. Together with the analogous argument for  $\eta\in[-\ell,0]$, this yields $u=u^{\eta}$ for all $\eta\in[-\ell,\ell]$ and
hence $u$ does not depend on $y$. Since this contradicts $\partial_{y}u(x',y)\not\equiv 0$, case (i) cannot occur for all $\eta\in[-\ell,\ell]$.

\medskip

\uline{Step 2.} We now observe that, because of the continuity of $u$, (i) and (ii) are preserved under limits. In other words, if (i) holds for some sequence  $(\eta_k)_k$ with $\eta_k\to\eta^{\star}$ as $k\to\infty$, then $u=T^{\eta^{\star}}$ and the same holds true for condition (ii).

\medskip

\uline{Step 3.} Let
\begin{align*}
  M_i:=\{\eta\in [-\ell,\ell]\colon \text{(i) holds}\} \quad\text{and}\quad M_{ii}:=\{\eta\in [-\ell,\ell]\colon \text{(ii) holds}\}.
\end{align*}
As a consequence of Steps 1 and 2, we obtain that both $M_i$ and $M_{ii}$ are infinite. In this step we will show that there is a value $\eta^\star$ that is an accumulation point of $M_i$ and $M_{ii}$ and moreover that there exists a strictly increasing sequence $(\eta_k)_k$ with
\begin{align*}
  \eta_k\in M_i\quad\text{and}\quad \eta_k\to\eta^{\star}\;\text{for}\;k\to\infty.
\end{align*}
According to Step 1 and Proposition \ref{either}, there exist points $\eta_i$ such that (ii) does not hold and (i) does and $\eta_{ii}$ such that (i) does not hold and (ii) does. By periodicity we may assume that $\eta_i<\eta_{ii}$. According to Step 2, $\eta_i$ is not an accumulation point of $M_{ii}$ and hence (since $M_i\cup M_{ii}=[-\ell,\ell]$) is an accumulation point of $M_i$. Let
\begin{align*}
  \eta^\star:=\sup\{\text{accumulation points of }M_i\in[\eta_i,\eta_{ii}]\}.
\end{align*}
According to Step 2, $\eta^\star<\eta_{ii}$. Consequently we deduce that $\eta^\star$ is an accumulation point of $M_{ii}$, since otherwise its definition as supremum is contradicted. By construction, $\eta^\star$ can be reached as the limit of an increasing sequence of points $\eta_k\in M_i$.

According to Step 2, $\eta^\star\in M_i\cap M_{ii}$ and hence $u=u^{\eta^{\star}}$ in $\tphil$.

\medskip

\uline{Step 4.} We now address the monotonicity.
From $u=T^{\eta_k}u$, we have
\begin{eqnarray*}
u(x)&\geq& u(x^{\eta_k})\qquad\hbox{in}\quad\tphil'\times[\eta_k,\eta_k+\ell]\\
&=&
u(x',2\eta_k-y)\\
&=&
u(x',2\eta^{\star}-2\eta_k+y)\qquad\hbox{since}\:u=T^{\eta^{\star}}u.
\end{eqnarray*}
Thus
\begin{eqnarray*}
\frac{u(x',y+2(\eta^{\star}-\eta_k))-u(x',y)}{2(\eta^{\star}-\eta_k)}\leq 0.
\end{eqnarray*}
In the limit $k\to\infty$ this gives
\begin{eqnarray*}
\partial_{y}u(x',y)\leq 0\qquad\hbox{in}\quad\tphil'\times[\eta^{\star},\eta^{\star}+\ell].
\end{eqnarray*}
By the Strong Maximum Principle (Theorem~\ref{t:strongmax}), this implies
\begin{eqnarray*}
\partial_{y}u(x',y)< 0\qquad\hbox{in}\quad\tphil'\times(\eta^{\star},\eta^{\star}+\ell).
\end{eqnarray*}
\end{proof}


\end{document}